\documentclass[12pt]{amsart}
\usepackage[french, english]{babel}

\usepackage{amscd,amssymb,epsfig,times,fullpage}
\newtheorem{thm}{Theorem}[section]
\newtheorem{prop}[thm]{Proposition}
\newtheorem{lem}[thm]{Lemma}
\newtheorem{cor}[thm]{Corollary}

\theoremstyle{remark}

\newtheorem{rem}[thm]{Remark}
\newtheorem{remark}[thm]{Remark}

\theoremstyle{definition}
\newtheorem{defn}[thm]{Definition}

\DeclareMathOperator{\Spc}{Spin^c}

\newcommand{\dist}{dist}

\newcommand{\iso}{\cong}

\newcommand{\C}{\mathbb{ C}}

\newcommand{\N}{{\mathbb {N}}}
\newcommand{\Z}{\mathbb{ Z}}

\newcommand{\lra}{\longrightarrow}

\newcommand{\BB}{\mathbb{ B}}

\newcommand{\Q}{\mathbb{ Q}}
\newcommand{\R}{\mathbb{ R}}
\newcommand{\KK}{\mathbb{ K}}

\newcommand{\ra}{{\rightarrow}}

\date{29 February 2008; MSC 2000: primary 53C23; secondary 55N99, 19K35, 19K56}

\begin{document}

\title{Coarse topology, enlargeability, and essentialness}
\author{B. Hanke}
\address{Mathematisches Institut, Ludwig-Maximilians-Universit\"at M\"unchen,
Theresienstr.~39, 80333 M\"unchen, Germany}
\email{hanke{\char'100}math.lmu.de}
\author{D. Kotschick}
\address{Mathematisches Institut, Ludwig-Maximilians-Universit\"at M\"unchen,
Theresienstr.~39, 80333 M\"unchen, Germany}
\email{dieter{\char'100}member.ams.org}
\author{J. Roe}
\address{Department of Mathematics, Pennsylvania State University, 
107 McAllister Building, University Park, PA 16802, USA}
\email{roe{\char'100}math.psu.edu}
\author{T. Schick}
\address{Mathematisches Institut, Universit\"at G\"ottingen, 
Bunsenstr.~3, 37073 G\"ottingen, Germany}
\email{schick{\char'100}uni-math.gwdg.de}

\begin{abstract}
Using methods from coarse topology we show that fundamental classes of closed enlargeable manifolds 
map non-trivially both to the rational homology of their fundamental  groups and to the $K$-theory of the
corresponding reduced $C^*$-algebras. Our proofs do not depend on the Baum--Connes conjecture and 
provide independent confirmation for specific predictions derived from this conjecture. 
\end{abstract}

\maketitle

\begin{otherlanguage}{french}

\begin{center}
{\bf Topologie \`a grande \'echelle, agrandissabilit\'e et non-annulation en homologie}
\end{center}

\begin{abstract}
En utilisant des m\'ethodes de  topologie \`a grande \'echelle, on prouve que les classes fondamentales
des vari\'et\'es agrandissables ne s'annulent pas, ni dans l'homologie rationelle de leurs 
groupes fondamentaux, ni dans la $K$-th\'eorie des $C^*$-alg\`ebres r\'eduites correspondantes. 
Nos r\'esultats ne d\'ependent pas de la conjecture de Baum--Connes, et confirment 
de fa\c con ind\'ependante certaines cons\'equences de cette conjecture.
\end{abstract}

\end{otherlanguage}

\bigskip


\section{Introduction and statement  of results}\label{s:intro}

In this paper we use methods from coarse topology to prove certain homological properties of enlargeable manifolds. 
The defining property of this class of manifolds is that they admit covering spaces that are uniformly large in all directions.
The intuitive geometric meaning of enlargeability is naturally captured by concepts of coarse topology, in particular by 
the notion of macroscopic largeness. We proceed by showing that enlargeability implies macroscopic largeness, which 
in turn implies homological statements in classical, rather than coarse, algebraic topology.

Using completely different methods, related results were previously proved in~\cite{HS,HS2}. We shall discuss the 
comparison between the two approaches later in this introduction, after setting up some of the terminology to
be used. Suffice it to say for now that our results here, unlike those of~\cite{HS,HS2}, are relevant to the Baum--Connes
conjecture for the reduced group $C^*$-algebra, in that we  verify specific predictions derived from this conjecture. 

\subsection*{Enlargeability}\label{ss:large}

Several versions of the notion of enlargeability or hypersphericity were introduced by Gromov and Lawson 
in~\cite{GL,GL2}. Here is the basic definition:
\begin{defn}\label{d:enl}
A closed oriented manifold $M$ of dimension $n$ is called {\it enlargeable} if for every $\epsilon >0$ there is a 
covering space $M_{\epsilon}\longrightarrow M$ that admits an $\epsilon$-contracting map 
\[
   f_{\epsilon}\colon M_{\epsilon} \longrightarrow (S^n, g_{can}) 
\]
to the $n$-sphere with its canonical metric, which is constant outside a compact set, and is of nonzero degree. 
\end{defn}
Here all covering spaces $M_{\epsilon}$ are given the pullback metrics induced by an arbitrary metric
on $M$. The choice of metric on $M$ matters only in that it has to be independent of $\epsilon$.

A variation on this definition is obtained by restricting the kind of covering space allowed for 
$M_{\epsilon}$. We shall call $M$ {\it universally enlargeable} if it is enlargeable and for all $\epsilon$ the 
covering $M_{\epsilon}$ can be taken to  be the universal covering $\widetilde{M}\longrightarrow M$. We shall
call $M$ {\it compactly enlargeable} if it is enlargeable and all $M_{\epsilon}$ can be taken to be compact,
equivalently to be finite-sheeted coverings.

\subsection*{Essentialness}\label{ss:ess}

Recall that Gromov~\cite{BC} called a closed oriented manifold $M$ {\it essential} if its fundamental class maps 
non-trivially to the rational homology of $B\pi_1(M)$ under the classifying map of its universal cover.  
It is natural to extend this definition to more general situations. For any homology theory $E$, we say
that an $E$-oriented manifold $M$ is {\it $E$-essential} if its orientation class maps non-trivially to $E_*(B\pi_1(M))$
under the classifying map of the universal covering.

In the context of coarse topology, one replaces the usual orientation class of $M$ by the orientation class
of the universal covering $\widetilde{M}$ in the coarse homology $HX_*(\widetilde{M})$, see Section~\ref{homology} below. 
Passing to the coarse homology of the universal covering is a procedure not unlike passing from $M$ to the classifying space of its
fundamental group, and the coarse fundamental class $[\widetilde{M}]_X$ may well vanish. We shall say that a manifold
$M$ (or its universal covering) is {\it macroscopically large} if it is essential for coarse homology, i.~e.~if 
$[\widetilde{M}]_X \neq 0\in HX_*(\widetilde{M})$.
In fact, Gromov suggested various versions of macroscopic largeness in~\cite{Large,AI,Macro}, and this definition, 
taken from~\cite{Gong}, is just one particular way of formalizing the concept.

We can now state our first main result.

\begin{thm}\label{t:first} 
{\rm (1)} Universally enlargeable manifolds are macroscopically large.

{\rm (2)} Macroscopically large manifolds are essential in rational homology.
\end{thm}

There are results by Dranishnikov~\cite{D} addressing the  converse to the first part of this theorem. 
He has shown that other notions of macroscopic largeness sometimes imply versions of enlargeability.

Combining the two implications in Theorem~\ref{t:first}, we obtain:

\begin{cor}\label{c:ee}
Universally enlargeable manifolds are essential in rational homology.
\end{cor}

That compactly enlargeable manifolds are essential was conjectured by Burghelea~\cite[Problem 11.1]{Schu}
quite some time ago and was proved fairly recently by Hanke and Schick~\cite{HS}, using index theory and 
the $K$-theory of $C^*$-algebras. One of the motivations for the present paper was the wish to give a direct 
and elementary proof of such a result, which does not use index theory and $K$-theory. After we achieved this
goal by finding the proof of Theorem~\ref{t:first} given in Section~\ref{proof1} below, it turned out that the 
sophisticated methods of~\cite{HS} can also be adapted to the consideration of infinite covers~\cite{HS2}.

While the ideas involved in our proof of Theorem~\ref{t:first} are indeed geometric and elementary, they do fit 
naturally into the framework of coarse homology, which we recall in Section~\ref{homology} following the 
books~\cite{R,Roe}. Our argument makes essential use of the coarse space 
\[
   B^n = [0,\infty) \bigcup_{\{1,2,3,\ldots\}} \left( \cup_i S^n(i) \right) \ .
\]
This balloon space, sketched in Figure~\ref{fig:connectedballoons}, is a coarse analogue of the one-point union.
It is defined using a collection of $n$-spheres of increasing radii $i = 1,2,3, \ldots$, with the sphere of radius $i$ 
attached to the point $i \in [0, \infty)$ at the south pole of $S^n$, and is equipped with the path metric. 

\begin{figure}
\begin{center}
\epsfig{file=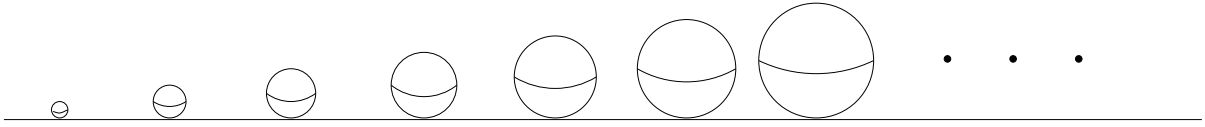,width=16cm}
\caption{The connected balloon space $B^n$}\label{fig:connectedballoons}
\end{center}
\end{figure}

The enlargeability assumption will be used to construct a coarse map  
\[
   \widetilde{M} \longrightarrow B^n
\]
that sends the coarse fundamental class of $\widetilde{M}$ to a nonzero class in the 
coarse homology of $B^n$ (see Proposition~\ref{p:central}). After this has been 
established, the proof of Theorem~\ref{t:first} can be completed quite easily.

\subsection*{Applications to the Baum--Connes map} 

After giving the proof of Theorem~\ref{t:first}, we proceed to use coarse topology to study the relation between enlargeability and the 
Baum--Connes assembly map in complex $K$-theory. This will lead us to some novel results on the Baum--Connes map that are 
interesting both in their own right and because of what they say about the relationship between various obstructions to the existence
of positive scalar curvature metrics. 

To formulate our results we  make the following definition.
\begin{defn}\label{BC} 
A closed $K$-theory oriented manifold $M$ is {\it Baum--Connes essential} if the image of its $K$-theoretic fundamental 
class under the composite map
\[
     K_*(M) \stackrel{c_*}{\longrightarrow} K_*(B \pi_1(M)) 
     \stackrel{\mu}{\longrightarrow}  K_*(C^*_{red} \pi_1(M))  
\]
is non-zero. Here, $c\colon M \longrightarrow B \pi_1(M)$ classifies 
the universal covering of $M$, and $\mu$ is the Baum--Connes assembly map. 
\end{defn}
In contrast to~\cite{HS,HS2}, we will work with the {\em reduced} group $C^*$-algebra throughout the present paper. 
We use the letter $K$ for the compactly supported complex $K$-homology defined by the $K$-theory spectrum. 
This is different from the 
convention in~\cite{R}, where $K_*$ denotes the analytically defined, hence  locally finite $K$-homology.  

Recall that a smooth manifold $M$ is orientable with respect to $K$-theory if and only if its tangent bundle admits a 
$\Spc$-structure. If $M$ is compact, then any choice of $\Spc$-structure determines a fundamental class $[M]$ in 
$K$-homology given by the corresponding Dirac operator, cf.~\cite[Chapter 11]{HR}. The image
\[
  \alpha(M) =  \mu\circ c_*([M]) \in K_*(C^*_{red} \pi_1(M)) 
\]
is given by the index of the $\Spc$ Dirac operator on $M$ twisted by the flat Hilbert module bundle 
\[
    \widetilde{M} \times_{\pi_1(M)} C^*_{red} \pi_1(M) \longrightarrow M 
\]
on $M$, as can be seen for example by a description of the Baum--Connes 
assembly map via Kasparov's $KK$-theory; cf.~\cite{Bl}. 

If a $\Spc$-structure on $M$ is induced by a spin structure, then the above construction can also be performed in real 
$K$-theory, leading to $\alpha_{\R}(M)\in KO_*(C^*_{red} \pi_1(M))$. In this case $\alpha (M)$ is the image of 
$\alpha_{\R}(M)$ under complexification. The Weitzenb{\"o}ck formula for the spin Dirac operator implies via the 
Lichnerowicz argument that if $M$ endowed with the fundamental class of a spin structure is Baum--Connes essential, 
then it does not admit a metric of positive scalar curvature. The Gromov--Lawson--Rosenberg conjecture predicts that 
the vanishing of $\alpha_{\R}(M)$ on a closed spin manifold $M$ is not only necessary, but also sufficient for the 
existence of a positive scalar curvature metric on $M$. Although this conjecture does not hold in general~\cite{dwyer,Schick}, 
it is expected that $\alpha_{\R}(M)$ captures all index-theoretic obstructions to the existence of a positive scalar curvature 
metric on $M$. This expectation is based in part on the relationship between the Gromov--Lawson--Rosenberg
conjecture and the Baum--Connes conjecture.

Recall~\cite{BCH} that the Baum--Connes conjecture claims that for any discrete group $\Gamma$, the 
assembly map 
\[ 
     K^\Gamma_*(\underline{E} \Gamma) \longrightarrow K_*(C^*_{red} \Gamma)
\]
is an isomorphism, where $\underline{E} \Gamma$ is the universal space for
proper $\Gamma$-actions, and $K_*$ denotes K-homology with compact
supports. The assembly map 
\[
   \mu\colon  K_*(B \Gamma) \longrightarrow K_*(C^*_{red} \Gamma) 
\]
considered above factors as 
\[
    K_*(B \Gamma) \stackrel{\cong}{\longrightarrow} K_*^\Gamma(E\Gamma) \stackrel{\gamma}{\longrightarrow} 
    K_*^\Gamma(\underline{E} \Gamma) \longrightarrow K_*(C^*_{red} \Gamma) \ ,
\]
where the first map is the canonical isomorphism between the equivariant $K$-theory of the free 
$\Gamma$-space $E\Gamma$ and the $K$-theory of the quotient $B\Gamma$ and $\gamma$ is 
induced by the canonical map $E\Gamma \to \underline{E} \Gamma$. Stolz~\cite{Ststable} has proved 
that if $M$ is spin and the Baum--Connes conjecture holds for the group $\pi_1(M)$, then the vanishing of $\alpha_{\R}(M)$ 
is sufficient for $M$ to stably admit a metric of positive scalar curvature. Here ``stably'' means that one 
allows the replacement of $M$ by its product with many copies of a Bott manifold $B$, which is any simply 
connected $8$-dimensional spin manifold with $\hat{A}(B)=1$. This result can be regarded as an instance of 
the universal nature of the index obstruction $\alpha_{\R}(M)$. 

It can be shown that the map $\gamma$ is rationally injective. Hence, if the Baum--Connes conjecture is true for 
$\pi_1(M)$, then $M$ is Baum--Connes essential if, and clearly only if, it is $K$-theory essential (for $K$-theory with
rational coefficients). 
In this direction, we shall prove the following unconditional result. We do not assume the Baum--Connes conjecture, 
which has not yet been verified in full generality and is actually expected to fail for some wild groups.

\begin{thm}\label{main} 
Every closed universally enlargeable $\Spc$-manifold is Baum--Connes essential.
\end{thm} 

In Section~\ref{s:BC} we shall discuss the coarse version of the Baum--Connes conjecture. It will become clear that
if this were known to be true, then Theorem~\ref{main} would follow from the first part of Theorem~\ref{t:first}. 
However, our proof of Theorem~\ref{main} will bypass this issue.

In the spin case, Theorem~\ref{main} essentially shows that the Gromov--Lawson obstruction~\cite{GL,GL2} to 
the existence of positive scalar curvature provided by enlargeability is subsumed by the index-theoretic obstruction 
$\alpha_{\R}(M)$, and even by $\alpha (M)$. Hanke and Schick previously proved this for the corresponding invariant 
in the $K$-theory of the {\em maximal} $C^*$-algebra of $\pi_1(M)$; see~\cite[Theorem~1.2]{HS} and~\cite{HS2}. 
Our result here neither implies nor is it implied by that of~\cite{HS,HS2}. On the one hand, the canonical map 
\[
   K_*(C_{max}^* \pi_1(M)) \longrightarrow K_*(C_{red}^* \pi_1(M))
\]
is not always injective, so that our conclusion here is stronger than the one in~\cite{HS,HS2}. On the other 
hand, we also use a stronger assumption, enlargeability, which implies the assumption of area-enlargeability used
in~\cite{HS,HS2}. (For area-enlargeability, the $\epsilon$-contracting property of  $f_{\epsilon}$ is required not for 
lengths, but only for two-dimensional areas.)

Having shown that the enlargeability obstruction to the existence of positive scalar curvature metrics is indeed subsumed by 
the universal index obstruction $\alpha_{\R}(M)$, we want to go further and also show this for the obstruction derived from
$\hat{A}$-enlargeability in the sense of Gromov and Lawson~\cite{GL,GL2}. We slightly generalize this concept by considering
the following amalgamation of enlargeability and $K$-theoretic essentialness:
\begin{defn}\label{d:Aenl} 
A closed $\Spc$-manifold $M$ is called {\it $K$-theory enlargeable}
if there is an $n\in\N$ such that for every $\epsilon >0$ there is a covering space
$M_{\epsilon}\longrightarrow M$ that admits an $\epsilon$-contracting map 
\[
   f_{\epsilon}\colon M_{\epsilon} \longrightarrow (S^n, g_{can}) 
\]
to the $n$-sphere with its canonical metric, which maps the complement of a compact set to the base point (equal to 
the south pole, say)  $S\in S^n$, 
and sends the $K$-theoretic orientation class to a non-trivial element of $K_*(S^n,S)$.
\end{defn}
When $n$ is the dimension of $M$, this definition reduces to Definition~\ref{d:enl}.
We shall prove the following generalization of Theorem~\ref{main}:
\begin{thm}\label{t:main} 
Every $K$-theory enlargeable $\Spc$-manifold is Baum-Connes essential.
\end{thm}
Note that in addition to considering $K$-theory enlargeability, we now also allow arbitrary covering spaces
in the definition of enlargeability, whereas in Theorem~\ref{main} we only used the universal covering.

If the $\Spc$ structure under consideration is induced by a spin structure, then the image of the $K$-theoretic
fundamental class under $f_{\epsilon}$ is given by the $\hat{A}$-genus of a regular fiber, and $K$-theory
enlargeability reduces to $\hat{A}$-enlargeability in the sense of Gromov and Lawson. The conclusion of 
Theorem~\ref{t:main} means that $\alpha(M)$ does not vanish. {\it A fortiori} the real index $\alpha_{\R}(M)$
does not vanish either, so that the obstruction to the existence of a positive scalar curvature metric on $M$
provided by $\hat{A}$-enlargeability is completely subsumed by $\alpha_{\R}(M)$.

Just like Theorem~\ref{main}, we prove Theorem~\ref{t:main} unconditionally, without assuming any unproved 
version of the Baum--Connes conjecture. However, if the Baum--Connes conjecture does hold for $\pi_1(M)$, 
then the special case of this theorem for $\Spc$-structures induced by spin structures can be derived from the 
result of Stolz mentioned above, because $\hat{A}$-enlargeability is preserved by stabilisation with the Bott manifold.

Although Theorem~\ref{main} is a special case of Theorem~\ref{t:main}, we shall first give a proof of this special 
case using the same ideas as in the proof of Theorem~\ref{t:first}. Then, in order to prove Theorem~\ref{t:main} in full 
generality, we will face serious additional complications explained in Section~\ref{s:coarsecoeff} below. The more 
straightforward proof we give for the special case of Theorem~\ref{main} does have another advantage in addition to
its simplicity, which is that it allows us to derive the following:
\begin{cor}\label{c:psc}
If $M$ is universally enlargeable and its universal covering $\widetilde{M}$ is spin, then $\widetilde{M}$ does not admit 
a metric with uniformly positive scalar curvature which is quasi-isometric to a pullback metric from $M$ via the identity.
\end{cor}
In the more general situation considered in Theorem~\ref{t:main}, the corresponding statement is not true. In fact,
Block and Weinberger~\cite{BW} give examples of spin manifolds $M$ which are $\hat{A}$-enlargeable (with $n=0$), but 
whose universal coverings do admit positive scalar curvature metrics that are quasi-isometric to pullback metrics from $M$.
Corollary~\ref{c:psc} shows in particular that examples of the kind considered in~\cite{BW} can never be universally
enlargeable.

\subsection*{Acknowledgements}
DK would like to thank S.~Weinberger, who in conversations in 2004 suggested detecting essentialness of enlargeable 
manifolds by mapping to (a coarse analogue of) a one-point union of spheres. BH, DK and TS are members of the DFG Priority 
Program in Global Differential Geometry.

\section{Coarse homology}\label{homology}

In this section we recall a few salient features of coarse homology, which we need for the proof of Theorem~\ref{t:first}. 
Our reference is~\cite{R}, see also~\cite{Roe}.

Let $M$ be a topological space. Its locally finite homology $H^{lf}_*(M;\Q)$ 
is the homology of the chain complex  $(C_*^{lf}(M), \partial)$, where $C_i^{lf}(M)$ 
is the abelian group of infinite rational linear combinations 
\[
     \sum_{\sigma} \alpha_{\sigma} \cdot \sigma
\]
of singular $i$-simplices $\sigma \colon\Delta^i \to M$ 
with the property that each compact set in $M$ is met by 
only finitely many simplices. 

This locally finite homology theory is a functor on the category of topological spaces and proper continuous maps.

Now let $M$ be a proper metric space. The coarse homology $HX_*(M;\Q)$ is
defined as follows. Let $\mathfrak{U}_i$ be a coarsening sequence 
of covers of $M$ in the sense of~\cite[p.~15]{R}. We then set
\[   
    HX_*(M) := \lim_{\ra_i} H^{lf}_*(|\mathfrak{U}_i|) \, , 
\]
where $|\mathfrak{U}_i|$ is the geometric realization of the nerve of $\mathfrak{U}_i$.
Coarse homology is functorial for coarse maps, which in the case of length spaces are 
precisely the proper maps which are large scale Lipschitz, see the definitions in~\cite[p.~9]{R}.

Lastly,  let $M^n$ be a closed oriented manifold. Fix a metric 
and a triangulation of $M$. The sum of all the lifts 
(with induced orientations) of the $n$-simplices in $M$ with respect to 
the projection $\widetilde{M} \longrightarrow M$ defines the locally finite fundamental class 
\[  
    [ \widetilde{M} ]_{lf} \in H^{lf}_n(\widetilde{M}) \, . 
\]
Note that there is a canonical map  
\[
    H^{lf}_*(\widetilde{M}) \longrightarrow HX_*(\widetilde{M}) \, . 
\]
Indeed, let $d$ be the maximal diameter of simplices in  $M$. Then the set 
of all open balls of radius $2d$ around each vertex in $\widetilde{M}$ defines an 
open cover $\mathfrak{U}$ of $\widetilde{M}$ which we can use as a particular 
cover in the coarsening sequence $(\mathfrak{U}_i)$ in the definition of coarse 
homology. It follows from the definition of the geometric realization $|\mathfrak{U}|$
that the simplicial complex $\widetilde{M}$ has a natural simplicial map to
$|\mathfrak{U}|$ and the above map is simply the composition
\[ 
    H^{lf}_*(\widetilde{M}) \longrightarrow H^{lf}_*(|\mathfrak{U}|) \longrightarrow HX_*(\widetilde{M}) \, . 
\]
The coarse fundamental class $[\widetilde{M}]_X$ of $\widetilde{M}$ is defined as the image of $[\widetilde{M}]_{lf}$ under 
this map. 

This construction actually applies not just to universal covers of closed manifolds,
but to all complete manifolds with bounded geometry.
Using it, we now define macroscopic largeness.
\begin{defn} [\cite{Gong}]
A complete oriented Riemannian manifold $N$ with bounded geometry 
is called {\it macroscopically large} if 
\[
    [N]_X \neq 0 \in HX_n(N) \, ,
\]
where $HX_*$ denotes the coarse homology.
\end{defn}

As an example for the calculation of coarse homology, fix a natural number $n$ and 
consider the balloon space 
\[
   B^n = [0,\infty) \bigcup_{\{1,2,3,\ldots\}} \left( \cup_i S^n(i) \right) \ .
\]
\begin{prop}\label{p:B}
The coarse homology of $B^n$ with $\mathbb{Q}$-coefficients in degree $n$ is given  by
\[
   HX_n(B^n) \cong \left( \prod_{i=1}^{\infty} \Q\right) / \left(\bigoplus_{i=1}^{\infty} 
  \Q \right) \, . 
\]
\end{prop}
\begin{proof} 
For $1 \leq l < \infty$, let 
\[
     B^n_{\geq l} \subset B^n
\]
be the subspace defined by removing the first $l$ spheres in 
$B^n$. We obtain a directed system 
\[
    B^n = B^n_{\geq 1} \to B^n_{\geq 2} \to 
                   B^n_{\geq 3} \to \ldots
\]
where each map 
\[
   B^n_{\geq l} \to B^n_{\geq (l+1)} 
\]
collapses the $l^{\textrm{th}}$ sphere onto the point $l \in [0,\infty)$. Then (by use 
of an appropriate coarsening sequence for $B^n$), 
$HX_n(B^n)$ can be calculated as 
\[
  \lim_{\longrightarrow_l} H^{lf}_n (B^n_{\geq l})  \cong   \lim_{\longrightarrow_l} 
              \left(  \prod_{i=1}^{\infty} \Q \right) / \left( 
               \bigoplus_{i=1}^l \Q \right) = \left( \prod_{i=1}^{\infty} \Q\right) / \left(\bigoplus_{i=1}^{\infty} 
  \Q \right) \, . 
\]
\end{proof}

\section{Proof of Theorem~\ref{t:first}} \label{proof1}

The first implication in Theorem~\ref{t:first} is a consequence of the following:
\begin{prop}\label{p:central}
Let $M$ be a closed oriented $n$-dimensional manifold which is universally enlargeable.
Then there is a coarse map 
\[
  \phi \colon \widetilde{M} \longrightarrow B^n
\]
such that for each $i \in \mathbb{N}$ the composition 
\[
    \widetilde{M}  \stackrel{\phi}{\longrightarrow} B^n \longrightarrow S^n(i)
\]
has degree $d_i\neq 0$. Furthermore, 
\[
   \phi_*([\widetilde{M}]_X) = (d_1, d_2, \ldots) \in \left( \prod_{i=1}^{\infty} \Q \right) / \left( \bigoplus_{i=1}^{\infty} \Q\right) \cong HX_n(B^n) \, . 
\]
In particular, $\widetilde M$ is macroscopically large.
\end{prop}
\begin{proof} 
Pick a Riemannian metric $g$ on $M$. We construct a cover of $\widetilde M$ by a sequence of compact 
balls
\[
      B_1 \subset B_2 \subset B_3 \subset \ldots \subset \widetilde M
\]
as well as a sequence of $1$-contracting maps 
\[
     f_{i} : \widetilde M \to S^n(i) \subset B^n \, , i = 1,2,3, \ldots 
\]
as follows. Set $B_0 := \emptyset$ and assume that $B_i$ has been
constructed. Because $M$ is enlargeable, there is a $1$-contracting map 
\[
   f_{i+1} \colon (\widetilde M , \widetilde g) \to S^n(i+1)
\]
which is constant (mapping to the basepoint in $S^n$) outside a compact subset 
$K_{i+1} \subset \widetilde M$ and of non-zero degree $d_{i+1}$. By precomposing 
$f_{i+1}$ with a deck transformation of $\widetilde M$ if necessary, we can assume that 
\[
    \dist (B_i, K_{i+1}) \geq 1 \, . 
\]
For $i=0$, this condition is empty. Now for $B_{i+1}$ we choose a closed ball containing 
$K_{i+1} \cup B_i$ and such that $\dist(\widetilde M \setminus B_{i+1} , B_i ) \geq 1$.
Then we define
\[
   \phi(x) = \left\{ \begin{array}{l} f_{i+1}(x), {\rm~if~} x \in K_{i+1} \, ,  \\
                              i + \dist(x, B_i), {\rm~if~} 0 < \dist(x, B_i) \leq 1   \, , \\ 
                           i+1, {\rm~if~} x \in B_{i+1} \setminus  K_{i+1}
                                {\rm~and~} \dist(x, B_i) \geq 1 \, . 
                    \end{array} \right. 
\]
The map $\phi$ is proper by definition, and it is large scale Lipschitz also by definition and because 
each $f_i$ is $1$-contracting. Thus $\phi$ is indeed a coarse map.
 
The claim about the image of the fundamental class follows from the
calculation of $HX_n(B^n)$ in Proposition~\ref{p:B}.
\end{proof}
This proof makes clear why we use the space $B^n$, rather than the one-point union of the spheres--
using the latter would not give us a proper map. 

The following is the contraposition of the second part of Theorem~\ref{t:first}.
\begin{prop}
If $M$ is not essential, then $\widetilde M$ is not macroscopically large.
\end{prop} 
\begin{proof}
Assume that $M$ is not essential. Then there is a finite subcomplex $S \subset B \pi_1(M)$ such that $c(M) \subset S$, 
the inclusion $S \subset B \pi_1(M)$ induces an isomorphism on fundamental groups, and $c_*([M]) = 0 \in H_n(S; \Q)$. 
Note that we use ordinary homology. Hence, if a chain in $B \pi_1(M)$ is a boundary, it is also 
a boundary in some finite subcomplex of $B \pi_1(M) $. 

We may assume without loss of generality that $M$ is a finite subcomplex of $S$ and 
$c\colon M \longrightarrow S$ is the inclusion. We choose a metric on  $S$. The induced inclusion 
\[   
   \widetilde{M} \stackrel{\widetilde{c}}{\longrightarrow} \widetilde{S}
\]
is then a coarse equivalence. 

We claim that the induced map 
\[
   HX_*(\widetilde{M})   \stackrel{\widetilde{c}_*}{\longrightarrow} HX_*(\widetilde{S})  
\]
sends $[\widetilde{M}]_X$ to zero. This is true for the following reason. Let the simplicial 
chain $C_M \in C_n(M;\Q)$ represent the fundamental class of $M$, and let $b\in C_{n+1}(S;\Q)$ 
be a simplicial chain with $\partial b =  C_M$. This exists by the choice of $S$.
As above we find an open cover $\mathfrak{U}$ of 
$\widetilde{S}$ by open balls with radius $2d$ (where $d$ is the maximal diameter of the 
simplices in $S$) so that $\widetilde{S}$ is a subcomplex of $|\mathfrak{U}|$. In particular, 
the chain $b$ induces a chain $\widetilde{b} \in C_{n+1}(|\mathfrak{U}|)$ (by lifting $b$ to 
$\widetilde{S} \subset |\mathfrak{U}|$) and the boundary of $\widetilde{b}$ is equal to 
$C_{\widetilde{M}}$, where $C_{\widetilde{M}}$ is the lift of $C_M$. Hence, we can indeed 
conclude that the homology class of $C_{\widetilde{M}}$ vanishes in $HX_n(\widetilde{S})$.  

The inclusion $\widetilde{c}$ is a coarse equivalence and so $[\widetilde{M}]_X \in HX_n(\widetilde{M})$ 
vanishes. Therefore $\widetilde{M}$ is not macroscopically large.
\end{proof}

\section{Proof of Theorem~\ref{main}}\label{s:BC}

In this section we use the ideas from the proof of Theorem~\ref{t:first} to prove Theorem~\ref{main} and
Corollary~\ref{c:psc}, using K-homology instead of ordinary homology.

The coarse Baum-Connes conjecture, see~\cite[Conjecture 8.2.]{R}, predicts that for a metric space $M$ 
of bounded geometry, the coarse assembly map 
\[
     \mu_{\infty} \colon KX_*(M) \longrightarrow K_*(C^*(M)) \ ,
\]
see~\cite[Chapter 8]{R}, is an isomorphism. Here, $KX_*(M)$ denotes coarse homology based on locally 
finite complex $K$-homology and $C^*(M)$ is the $C^*$-algebra of locally compact finite propagation operators 
on $M$, see~\cite[Definition 3.4.]{R}. 

At this point it is useful to remark that locally finite $K$-homology can be defined in two ways. On the one hand, the 
operator-theoretic description of $K$-homology, which is used in~\cite{R} for defining the assembly map $\mu_{\infty}$, 
leads directly to a locally finite theory. On the other hand, the paper~\cite{KKS} associates to any homology theory
$E$ a Steenrod homology theory $E_*^{st}$ defined on compact metric 
pairs. One then defines the locally finite $E$-homology $E_*^{lf}(M)$ of a locally compact metric space $M$ as the 
Steenrod $E$-homology of the compact pair $(M \cup \{ \infty\}, \{\infty\})$, where $M\cup \{\infty\}$ is the one-point 
compactification of $M$. If $M$ is a countable and locally finite CW complex and $E$ is ordinary homology, then this 
definition coincides with our previous definition from Section~\ref{homology}, see~\cite[Theorem 1]{Milnor}. In the 
case of $K$-homology, the two descriptions of the locally finite theory coincide by~\cite{KS}.

\begin{prop} \label{p:coarseBC} 
The balloon spaces $B^n$ satisfy the coarse Baum-Connes conjecture.
\end{prop} 
This can be seen by referring to a deep result of Yu~\cite{Yu} saying that the coarse Baum--Connes conjecture is true 
for metric spaces of bounded geometry which admit uniform embeddings into seperable complex Hilbert spaces. 
However, the spaces $B^n$ are simple enough to check the coarse Baum-Connes conjecture by a direct calculation. 
This will be performed later in this paper for slightly different spaces $\mathcal{B}^n$, see 
Proposition~\ref{p:coarseass}, and can be adapted easily to the $B^n$.

We now start the proof of Theorem~\ref{main}. Let $M$ be a closed universally enlargeable $\Spc$-manifold. We
set $\pi = \pi_1(M)$.   

The composition 
\[
     K_*(M) \stackrel{c_*}{\longrightarrow} K_*(B \pi) 
     \stackrel{\mu}{\longrightarrow}  K_*(C^*_{red} \pi) \, . 
\]
can alternatively be regarded as the composition 
\[
     K_*(M) \stackrel{PD}{=}  K_{* +1} (D_{\pi}^*(\widetilde{M}) / C_{\pi}^*(\widetilde{M})) 
     \stackrel{\partial}{\longrightarrow} K_*( C_{\pi}^*(\widetilde{M})) \iso K_*(C^*_{red} \pi) \, . 
\]
Here we work with the operator-theoretic description of $K$-theory and the assembly map from~\cite[Chapter 5]{R}
using a Paschke duality map $PD$ and a connecting homomorphism $\partial$. Recall from that reference that    
$D^*(\widetilde{M})$ is the $C^*$-algebra generated by all pseudolocal finite propagation operators 
on $\widetilde{M}$ and the subscript $\pi$ indicates that only $\pi$-invariant 
pseudolocal (or locally compact) finite propagation operators are contained in the  generating set.  

It is therefore enough to show that the composition 
\[
    K_*(M) \longrightarrow K_*(C_{\pi}^*(\widetilde{M})) \longrightarrow K_*(C^*(\widetilde{M}))
\]
sends the $K$-theoretic fundamental class of $M$ to a non-zero class. This composition can be factored through the transfer map 
\[
 tr \colon  K_*(M) = K_{*+1}(D_{\pi}^*(\widetilde{M})/C_{\pi}^*(\widetilde{M})) \longrightarrow 
 K_{*+1}(D^*(\widetilde{M})/C^*(\widetilde{M})) = K^{lf}_*(\widetilde{M}) \, . 
\]
This map simply forgets the $\pi$-action. 
We obtain a commutative diagram
\begin{equation}\label{eq:CD}
   \begin{CD} 
     K_*(M) @>>>       K_*(C^*(\widetilde{M}))  \\
      @Vtr VV                           @V=VV           \\
 K_*^{lf}(\widetilde{M}) @>>>  K_*(C^*(\widetilde{M})) \\
      @VVV                              @V=VV             \\
      KX_*(\widetilde{M})        @>\mu_{\infty} >> K_*(C^*(\widetilde{M}))\\
       @VV{\phi_*}V     @V{\phi_*}VV  \\
       KX_*(B^n) @>{\mu_\infty}>{\cong}> K_*(C^* (B^n))\, ,
\end{CD}
\end{equation}
where $\phi_*$ is induced by the coarse map $\phi\colon\widetilde M \longrightarrow B^n$ defined in Section~\ref{proof1}.

Because $\mu_\infty$ is 
an isomorphism for $B^n$ by Proposition~\ref{p:coarseBC}, it is enough to show that the 
image of the $K$-theoretic fundamental class $[M]_K$ of $M$ in $KX_*(B^n)$ is non-zero. 
This can be done exactly as in the proof of Theorem~\ref{t:first} (1), using $K$-homology instead of ordinary homology. 
Namely, one establishes that $KX_n(B^n)\iso (\prod \Z)/(\bigoplus\Z)$, and that again $[M]_K$ is mapped
to the sequence represented by the degrees of the $f_{1/i}$, which is by assumption non-zero in the quotient group.

This completes the proof of Theorem~\ref{main}. To prove Corollary~\ref{c:psc}, notice that if the universal covering
$\widetilde{M}$ is spin, then we can start the argument with the $K$-theoretic orientation class of the spin structure in the 
group $K_*^{lf}(\widetilde{M})$ in~\eqref{eq:CD}, rather than starting at $K_*(M)$ and applying the transfer. Then the usual Lichnerowicz 
argument and the coarse invariance of $K_*(C^*(\widetilde M))$ imply the conclusion.

\section{Coarse $C^*$-algebras with coefficients}\label{s:coarsecoeff}

\subsection{Outline of the proof of Theorem~\ref{t:main}} \label{s:outline} 

The proof of Theorem~\ref{t:main} is based on ideas similar to those in our proof of Theorems~\ref{t:first} and~\ref{main}. 
However, the discussion of manifolds that are enlargeable with respect to arbitrary coverings requires a substantial 
refinement of these methods because we cannot work with the universal covering only and have to use many different 
coverings at once. We do this by introducing the coarse space 
\[
   \mathcal{M}  :=  \bigcup_{i=1,2,3, \ldots} M_{1/i}
\]
where $M_{1/i}$ is the covering of $M$ with an $1/i$-contracting map to $S^n$
given by enlargeability, and where the distinct components are placed at
distance $\infty$ from each other. Note that this does
not mean that these components are independent from each other - for the coarse type 
of this space, global metric bounds (i.~e.~referring to all components at once) will be crucial.

The balloon space $B^n$ considered before now has to be replaced by the disjoint union 
\[  
       \mathcal{B}^n  :=  \bigcup_{i=1,2,3, \ldots} \underline{S}^n(i) \, ,
\]
where the notation $\underline{S}^n$ means that 
a whisker of infinite length is attached at the south pole of $S^n$, see Figure~\ref{fig:balloons}. 

\begin{figure}
\begin{center}
\epsfig{file=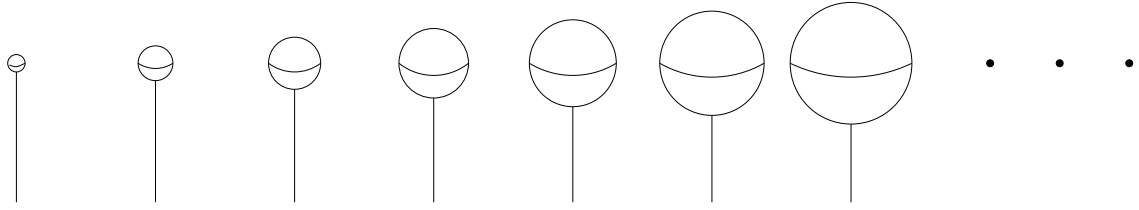,width=15cm}
\caption{The balloon space $ \mathcal{B}^n$, with infinite whiskers attached to each sphere.}\label{fig:balloons}
\end{center}
\end{figure}

In a similar fashion as in the proof of Theorem~\ref{t:first}, 
the fact that $M$ is $\hat{A}$-enlargeable leads to  a coarse map 
\[   
             \mathcal{M} \longrightarrow \mathcal{B}^n  
\]
which maps the coarse $K$-theoretic fundamental class of $\mathcal{M}$ to a nonzero 
class in the coarse $K$-homology of $\mathcal{B}^n$. 
The region outside the respective compact subset  in each $M_{1/i}$ 
(outside of which the map to $S^n$ is constant) is  mapped 
to the whisker in $\underline{S}^n(i)$ in order to get a proper map.

Our strategy for proving Theorem~\ref{t:main}  can now be 
outlined as follows. The information carried by 
\[
   \alpha(M) \in K_* ( C^*_{red} \pi_1(M) )
\]
will be split into two parts; one goes to the $K$-theory of the coarse algebra 
of $\mathcal{M}$ and the other is captured by introducing a collection 
of coefficient $C^*$-algebras for this coarse  algebra which is equivalent to 
$C^*_{red} \pi_1(M_{1/i})$ over $M_{1/i}$. The relevant element in the $K$-group 
of $C^*(\mathcal{M})$ is then sent to the $K$-theory 
of the coarse algebra of $\mathcal{B}^n$ with coefficients in the same 
collection of $C^*$-algebras. We will analyze this class by showing that the 
coarse assembly map with coefficients is an isomorphism for the space $\mathcal{B}^n$
and finally using a homological version of Atiyah's $L^2$-index theorem~\cite{MFA}. 

\subsection{Technical preliminaries for the proof of Theorem~\ref{t:main}}

Let $A$ be a $C^*$-algebra and let $M$ be a coarse metric space. 
\begin{defn} 
An {\it adequate} $M$-module with coefficients in $A$ is a Hilbert $A$-module $H$ 
together with a left $C_0(M)$-action, i.~e.~a $C^*$-homomorphism 
\[
   C_0(M) \longrightarrow \BB (H) \, , 
\]
where $\BB (H)$ is the $C^*$-algebra of adjointable $A$-module 
homomorphisms $H \to H$; see~\cite[Definition VI.13.2.1]{Bl}. In addition, 
the module $H$ is required to be sufficiently large in the  sense of~\cite[Defintion 4.5]{HPR}. 
\end{defn} 

For an adequate $M$-module $H$ the coarse algebra $C^*(M;A)$ is defined as the algebra 
generated by the locally $A$-compact finite propagation operators in $\BB (H)$. 
Note that by~\cite[Proposition 5.5]{HPR} this construction is functorial after composing with
the $K$-theory functor.

We define the coarse $K$-homology of $M$ with coefficients in $A$ as 
\[
    KX_* (M;A) := \lim_{\ra_i} KK_* (C_0 (|\mathfrak{U}_i|); A)  \, , 
\]
where $(|\mathfrak{U}_i|)_i$ is a coarsening sequence for $M$. Recall 
that $KK_*(C_0(M) ; A)$ is a locally finite theory. If $M$ is a manifold we
can start with a (fine) covering whose nerve is homeomorphic to $M$ and will
therefore obtain a canonical map $K^{lf}_*(M;A)\longrightarrow KX_*(M;A)$, called a
coarsening map.

In analogy with~\cite[Chapter 5]{R} we have an assembly map 
\[
  \mu \colon   KK_*(C_0(M); A) \longrightarrow K_*(C^*(M;A)) 
\]
which factors through $KX_*(M;A)$
and defines the coarse assembly map 
\[
   \mu_{\infty} \colon KX_*(M;A) \longrightarrow K_*(C^*(M;A)) \, ,
\]
compare p.~75 of~\cite{R}.

As discussed in Section~\ref{s:BC} above, the coarse Baum-Connes conjecture predicts that 
for a metric space $M$ of bounded geometry, the coarse assembly map 
\[
     \mu_{\infty} \colon KX_*(M ; \C)  \longrightarrow K_*(C^*(M)) 
\]
is an isomorphism. Here $C^*(M) = C^*(M ; \C)$ is the usual coarse algebra of locally compact finite 
propagation operators on $M$ as considered before. We shall now prove a generalized version of this 
conjecture with coefficients for the spaces 
\[
   \mathcal{B}^n  :=  \bigcup_{i=1,2,3, \ldots} \underline{S}^n(i)  
\]
introduced in Subsection~\ref{s:outline}. Recall that $\underline{S}^n$ denotes an $n$-sphere to which  
a whisker of infinite length has been attached at the south pole.
\begin{prop}\label{p:coarseass} 
Let $A$ be a $C^*$-algebra. Then the coarse assembly map 
\[
     KX_*(\mathcal{B}^n;A) \stackrel{\mu_{\infty}}{\longrightarrow} K_*(C^*(\mathcal{B}^n;A))
\]
is an isomorphism.
\end{prop} 
This Proposition plays the same role  in the proof of Theorem~\ref{t:main} as did Proposition~\ref{p:coarseBC} 
in the proof of Theorem~\ref{main}. The proof of Proposition~\ref{p:coarseass} is by induction on $n$ 
and uses a Mayer-Vietoris argument. Let us first collect some facts concerning Mayer-Vietoris sequences
for coarse $K$-theory and for $K$-theory of coarse algebras.
\begin{enumerate}
\item The K-theory of coarse $C^*$-algebras with coefficients has a
  Mayer--Vietoris sequence for coarsely excisive decompositions, by~\cite[Corollary 9.5]{HPR}. 
  Here a decomposition $X=X_1\cup X_2$ is {\it coarsely excisive} if for each $R>0$ there is an $S>0$ 
  such that the intersection of the $R$-neighborhoods of $X_1$ and $X_2$ is contained in the $S$-neighborhood 
  of $X_1\cap X_2$. 
\item The theory $KX_*(\cdot,A)$ has a Mayer--Vietoris sequence under the same assumptions as above.
  Recall that coarse $K$-homology with coefficients in $A$ is obtained by
  calculating $KK(C_0(|\mathfrak{U}_i|),A)$ and then passing to  a limit. We now use
  that $X\longmapsto KK(C_0(X),A)$ is a homology theory, and therefore has a Mayer--Vietoris
  sequence (compare~\cite[Section 5]{R} for a proof if $A=\mathbb{C}$), and
  observe that the coarse excisiveness implies that this
  Mayer--Vietoris sequence is compatible with coarsening (the intersection of the
  nerves of the coarsening sequences for $X_1$ and $X_2$ can be chosen to give a
  coarsening sequence for $X_1\cap X_2$). By passing to the direct limit, we
  obtain the required Mayer--Vietoris sequence in coarse K-homology. (Note that the direct
  limit functor is an exact functor). 
\item The two Mayer--Vietoris sequences are natural with respect to the coarse
  assembly map.
\end{enumerate}
In addition, we recall that a coarse space $X$ is called {\it flasque} if there is a coarse map 
$t\colon X\to X$ which is close to the identity (in the sense of coarse geometry) and such that
for each $R>0$ there is an $S>0$ such that $d(t^n(x),t^n(y))<S$ for each $n\in\mathbb{N}$ and each 
$x,y$ with $d(x,y)<R$, and such that for each compact set $K \subset X $ we have 
$t^n(K)\cap K=\emptyset$ for all sufficiently large $n$; cf.~\cite[Section 10]{HPR}.
\begin{lem}{\cite[Proposition 10.1]{HPR}} If a coarse space $X$ is flasque, then the K-theory of its 
coarse $C^*$-algebra with coefficients vanishes. 
\end{lem}

\begin{proof}[Proof of Proposition~\ref{p:coarseass}] 
We start by describing the induction step based on the above statements. 

Let $n > 0$ and consider the decomposition 
\[
    \mathcal{B}^n=\mathcal{D}_+^n\cup \mathcal{D}_-^n \, , 
\]
where $\mathcal{D}_+^n$ and $\mathcal{D}_-^n$ are the unions of the left and right hemispheres in 
the $n$-spheres contained in  $\mathcal{B}^n$, both with infinite whiskers attached at the south pole 
(considered as a point on the boundary of each of the two hemispheres). This decomposition is coarsely
excisive. By the Mayer--Vietoris principle, it suffices to prove that the coarse assembly map is an 
isomorphism for $\mathcal{D}_{\pm}^n$ and for their intersection. For the intersection, this is an immediate 
consequence of the induction hypothesis. Concerning the spaces $\mathcal{D}_{\pm}^n$, we will prove that 
the coarse $K$-theory  as well as the $K$-theory of the associated coarse $C^*$-algebras (both with coefficients 
$A$) vanish.  The last statement holds because the spaces $\mathcal{D}^n_{\pm}$ are flasque. The required map 
$t$ is  the composition of two maps.  The first one moves all points on the half spheres a unit distance to 
the center of the half sphere along great arcs connecting the boundary to this center (points close to the center are 
mapped to the center). The second one is  an isometric rotation, mapping this center two units toward 
the south pole; points which would be rotated out of the half sphere are mapped to the whisker, to the point with 
same distance to the south pole as would have been the distance of the rotated point to the south pole in the
full sphere. Finally, points on the whisker are just moved one unit further away from the south pole. 

Concerning the first statement, we observe that there is a coarsening sequence $(\mathfrak{U}_i)$ 
for $\mathcal{D}^n_{\pm}$, so that the first $i$ components of the nerve of $\mathfrak{U}_i$ 
are properly homotopy equivalent to a ray $[0,\infty)$. Because the locally finite $K$-homology 
of such a ray vanishes, we get indeed 
\[
   KX_*(\mathcal{D}^n_{\pm};A) = \lim_{\lra_i} K^{lf}_*(|\mathfrak{U}_i|;A) = 0 \, . 
\]
After finishing the induction step, it remains to verify that  
the assembly map is an isomorphism for $n=0$. In
this case, we have a disjoint union decomposition 
\[
     \mathcal{B}^0 = \bigcup_i \big( \{0\}\cup [2i,\infty) \big) 
\]
where the spaces $\{0\} \cup [2i, \infty)$ carry the metric induced from $[0,\infty)$.
Here, despite the fact that for different $i$ the subspaces $\{0\} \cup [2i, \infty) \subset \mathcal{B}^0$ 
are at distance infinity, it is crucial that the definition of the
coarse $C^*$-algebra requires uniform bounded propagation for its
operators. For each $k \in \N$, we define the space 
\[
     \mathcal{B}^0_k = \bigcup_{i = 1, \ldots, k-1} [0,\infty) \; \cup \; \bigcup_{i = k, k+1, \ldots} \big( \{0\} \cup [2i, \infty) \big) 
\]
which can be identified with the nerve of the $k$-th member of an appropriate coarsening sequence for $\mathcal{B}^0$. For 
each $k$, we have a canonical coarse equivalence
\[
      \mathcal{B}^0_k \to \mathcal{B}^0
\]
mapping $[0,2i)$ to $\{0\}$ for $i \leq k-1$. We need to show that the induced map 
\[
    \lim_{\lra_k} K^{lf}_*(\mathcal{B}^0_k, A) \to \lim_{\lra_k} K_*(C^*(\mathcal{B}^0_k,A)) \cong K_*(C^*(\mathcal{B}^0,A)) 
\]
is an isomorphism. Equivalently, we will prove that 
\[
    \lim_{\lra_k} K_*(D^*(\mathcal{B}^0_k,A)) = 0 \, . 
\]
Let $k$ be fixed and let us work with the adequate $\mathcal{B}^0_k$-module 
of $L^2$-function with values in the standard Hilbert $A$-module 
\[
    H_A = l^2(\N) \otimes A \, . 
\]
For each $l > k$,  we define a $C^*$-subalgebra 
\[
       D^*(\mathcal{B}^0_k,A)_l \subset D^*(\mathcal{B}^0_k,A) \, , 
\]
the closure of the set of all pseudolocal finite propagation operators that for $i \geq l$ do 
not interact between the two components of 
\[
        \{0\} \cup [2i,\infty) \subset \mathcal{B}^0_k \, . 
\]
Note that this condition is weaker than  restricting to 
operators of propagation less than $2l$ - this last condition does 
not define a subalgebra of $D^*(\mathcal{B}^0_k,A)$. Because $K$-theory commutes 
with direct limits, 
\[
    K_*(D^*(\mathcal{B}^0_k),A)  = \lim_{\lra_l} K_*(D^*(\mathcal{B}^0_k)_l,A) \, . 
\]
The last $K$-theory groups can be calculated explicitely. Note the canonical decomposition 
\[
   D^*(\mathcal{B}^0_k)_l = D^*(\Omega,A ) \times  \prod_{i \geq l} \BB(H_A)\, , 
\]
where 
\[
    \Omega = \bigcup_{i=1, \ldots, k-1} [0,\infty) \, \cup \, \bigcup_{i = k, \ldots, l-1} ( \{0\} \cup [2i,\infty)) 
      \, \cup \, \bigcup_{i =l, l+1, \ldots} [2i, \infty) \, . 
\]
The algebra $\prod_{i\geq l} \BB(H_A)$ is the multiplier algebra of $\prod_{i \geq l} \KK(H_A)$ and 
therefore has vanishing $K$-theory, cf.~\cite{WO}. The $K$-theory of the algebra $D^*(\Omega,A)$  
appears in the long exact sequence 
\[
    \ldots \ra K_*(C^*(\Omega,A)) \to K_*(D^*(\Omega,A)) \to K^{lf}_{*-1}(\Omega,A) \stackrel{\mu_{\infty}}{\to} 
               K_{*-1}(C^*(\Omega,A)) \to \ldots  
\]
and can therefore be studied by examining the coarse assembly map $\mu_{\infty}$. Now $\Omega$ is 
coarsely equivalent to the flasque space $\bigcup_{i \in \N} [0, \infty)$ and hence 
\[
    K_*(C^*(\Omega,A)) = 0 \, . 
\]
On the other hand, 
\[  
    K^{lf}_0(\Omega,A) = \prod_{k \leq i \leq l-1} \Z  \, , \qquad \qquad~~ K^{lf}_1(\Omega,A) = 0 \, , 
\]
by the fact that the locally finite $K$-homology of a disjoint union is the direct product 
of the locally finite $K$-homologies of the individual components. Altogether, we have 
\[
  K_1(D^*(\mathcal{B}^0_k,A)) = \lim_{\lra_l}  K_1(D^*(\mathcal{B}^0_k,A)_l) = \bigoplus_{i \geq k} \Z 
\]
and $K_0(D^*(\mathcal{B}^0_k),A) = 0$. We conclude
\[
  \lim_{\lra_k} K_1(D^*(\mathcal{B}^0_k,A)) = \lim_{\lra_k} \bigoplus_{i\geq k} \Z = 0 
\]
which finishes the proof of Proposition~\ref{p:coarseass}. 
\end{proof} 

For the proof of Theorem~\ref{t:main} in the next section we need a generalization of the construction 
underlying Lemma 5.14 in~\cite{R}. Let $M$ be a complete connected Riemannian manifold of positive 
dimension with fundamental group $\pi$. We denote by $\widetilde M$ the universal cover of $M$.
The adequate $C_0(\widetilde M)$-module $L^2(\widetilde M)$ carries an induced (right) unitary $\pi$-operation.
As usual, let $C_{\pi}^*(\widetilde M)$ be the $C^*$-algebra generated by locally compact $\pi$-invariant 
operators $L^2(\widetilde M)  \longrightarrow L^2( \widetilde M )$ of finite propagation. We furthermore consider 
the adequate $C_0(M)$-module $L^2(M,L)$, where 
\[
    L = \widetilde M \times_{\pi} C^*_{red} \pi 
\]
is the Mishchenko line bundle on $M$ (we equip $\widetilde M$ and $C^*_{red}\pi$ with the canonical 
right respectively left $\pi$-actions). Note that $L^2(M,L)$ is a Hilbert-$C^*_{red} \pi$-module in a canonical way. 
Denoting by $C^*(M, C^*_{red}\pi)$ the algebra generated by locally compact $C^*_{red}\pi$-linear finite propagation 
operators $L^2(M,L) \longrightarrow L^2(M,L)$ we wish to define a $C^*$-algebra map   
\[
    \psi \colon C^*_{\pi}(\widetilde M)  \longrightarrow C^*(M, C^*_{red}\pi) \, . 
\]
The construction of $\psi$ is almost tautological. Let 
\[
     T \colon L^2(\widetilde M) \longrightarrow L^2(\widetilde M) 
\]
be a $\pi$-invariant locally compact operator of propagation $R > 0$ and let
$\sigma\colon M \longrightarrow L$ be an $L^2$-section
of the Mishchenko line bundle. We can identify $\sigma$ with a
$\pi$-equivariant map
\[
     f \colon \widetilde M \longrightarrow C^*_{red} \pi \, . 
\]
Here $\pi$ acts on the right on $C^*_{red} \pi$ by $\xi \cdot \gamma := \gamma^{-1} \cdot \xi$. Let 
us assume for the moment that $f$ has image contained in $\C[\pi] \subset C^*_{red} \pi$. This map 
can be considered as a $\pi$-labeled family of maps $\widetilde M \longrightarrow \C$. To each of these
maps, the operator $T$ is applied. Because $T$ is $\pi$-equivariant, the resulting map 
\[
     T(f) \colon \widetilde M \longrightarrow \C[\pi]
\]
is again $\pi$-equivariant and hence induces a section of the bundle 
\[ 
    \widetilde M \times_{\pi} \C[\pi] \longrightarrow  M \, . 
\] 
The construction for general $f$ is by completion. By definition, the
section of $L \longrightarrow M$ obtained in this way is equal to $\psi(T)(\sigma)$. 
By construction, 
\[
   \psi(T) \colon L^2(M,L) \longrightarrow L^2(M,L) 
\]
is locally compact and has propagation less or equal to $R$. (Notice that part of the propagation may go in the 
$C^*_{red} \pi$-direction.) We emphasize that the map $\Psi$ is not in general an isomorphism 
of $C^*$-algebras,  because the propagation into the $C^*_{red} \pi$ direction is 
not required to be bounded in $C^*(M, C^*_{red}\pi)$. 

\begin{remark}\label{rem:non-free}
  Note that this construction also works if, instead of $\pi$, a subgroup $H$ acts freely
  on $\widetilde M$, and we work with the Hilbert $C^*_{red}\pi$-module bundle
  $\widetilde M\times_HC^*_{red}\pi$.
\end{remark}

\section{Proof of Theorem~\ref{t:main}}\label{s:proofmain}

By a suspension argument, we may restrict to the case of even $n$.   
For each $i \in \{1,2,3, \ldots \}$, choose a connected cover 
\[
    M_{1/i} \longrightarrow M 
\]
together with a $\frac{1}{i}$-contracting map 
\[
    f_{1/i} \colon M_{1/i} \longrightarrow S^n
\]
which is constant (with value equal to the south pole $S \in S^n$) 
outside a compact subset of $M_{1/i}$ and which is of nonzero $K$-theoretic degree $z_i \in \Z$.
The maps $f_{1/i}$ induce a coarse map 
\[
   \phi \colon \mathcal{M}=\amalg_{i\in\N} M_{1/i} \longrightarrow \mathcal{B}^n \ .
\]
Thanks to the whisker present in $\underline{S}^n$, we simply map a point on
the region outside the compact set, where $f_{1/i}$ is constant, to the point on the whisker
whose distance to the origin is the distance to the compact set. Further, we set  
\[
    \Gamma_i = \pi_1(M_{1/i}) 
\]
with respect to an arbitrary basepoint and consider the adequate $M_{1/i}$-module with 
coefficients in $C^*_{red} \Gamma_i$ 
defined by $L^2(M_{1/i}, L_i)$, the space of $L^2$-sections of the Mishchenko line bundle 
\[
       L_i = \widetilde M \times_{\Gamma_i} C^*_{red} \Gamma_i \longrightarrow M_{1/i} \, . 
\]
Recalling that $\widetilde{\mathcal{M}}$ is a 
disjoint union of copies of $\widetilde M$, we have a transfer map 
\[
    \Delta \colon C^*(\widetilde{M}) \longrightarrow C^*( \widetilde{ \mathcal{M} }  )
\]
induced by the diagonal embedding 
\[
     \BB( L^2(\widetilde{M} ))  \longrightarrow  \prod_i \BB (L^2 (\widetilde{M} )) \subset   
     \BB ( L^2( \widetilde{\mathcal{M} } )) \, .  
\]
This map restricts to a map between algebras of  locally compact 
operators of finite propagation and hence induces the map $\Delta$. 

Let
\[
       \Gamma  := \bigoplus_i \Gamma_i \subset \bigoplus_i \pi_1(M) \, . 
\] 
We define the Mishchenko bundle 
\[
  \mathcal{L} := \amalg_{i\in\N}\widetilde{{M}_{1/i}}  \times_{ \Gamma_i } C^*_{red}\Gamma
     \to \amalg_i M_{1/i}=\mathcal{M} \ .
\]   
Our argument uses $C^*_{red}\Gamma$
as a coefficient $C^*$-algebra where we work with the adequate $\mathcal{M}$-module
$L^2(\mathcal{M}, \mathcal{L})$. (Here we assume again that $\dim M > 0$ to
make sure that this module is adequate.)

\begin{rem} 
The heuristic meaning of this construction is that we would have to choose a coefficient
$C^*$-algebra on $\mathcal{M}$ that varies from component to component
and is equal to $C^*_{red} \Gamma_i$ over $M_{1/i}$; we artifically blow this
up to $C^*_{red}\Gamma$ to avoid the necessity to develop additional theory.
\end{rem} 

By performing the construction at the end of Section~\ref{s:coarsecoeff}, in
particular Remark~\ref{rem:non-free}, on each component of $\mathcal{M}$
separately, we get a map of $C^*$-algebras  
\[
 \psi \colon C_\Gamma^*(\widetilde{\mathcal{M}}) \longrightarrow
 C^*(\mathcal{M} ; C^*_{red} \Gamma)  \,    
 \]
as follows.
Fix an $L^2$-section $\sigma$ of the Mishchenko line bundle $\mathcal{L}$, which is a bundle
of free Hilbert-$C^*_{red}\Gamma$-modules of rank one over $\mathcal{M}$. The section $\sigma$ 
is the direct sum of sections $(\sigma_i)$, where $\sigma_i$ is the restriction to $\widetilde{M_{1/i}}$, a
section of the restriction of $L$ to the Mishchenko bundle over $M_{1/i}$.
Moreover, $\Gamma$ acts on $\widetilde{{M}_{1/i}}$ via the projection $\Gamma\to\Gamma_i$ with a 
free and discrete action of $\Gamma_i$. Consequently, Remark \ref{rem:non-free} applies and we can 
define $\psi(T)(\sigma)$ as the direct sum $\psi(T)(\sigma):=(\psi(T_i)(\sigma_i))_{i\in\N}$. With this definition
$\psi(T)$ has propagation $R$, and this gives the required homomorphism of $C^*$-algebras 
$$
\psi\colon C_\Gamma^*(\widetilde{\mathcal{M}})\longrightarrow
C^*(\mathcal{M};C^*_{red}\Gamma) \ .
$$

If $X$ is a topological space and $A$ a $C^*$-algebra, we define
$K^{lf}(X;A):= KK(C_0(X),A)$.

The proof of Theorem~\ref{t:main} now proceeds via the following commutative
diagram, where  we set $\pi:= \pi_1(M)$:
\begin{equation*}\label{eq:bigdiagram}
  \begin{CD}
    K_0(M) @>{\iso}>{PD}> K_1(D_\pi^*\widetilde M/C_\pi^*\widetilde M) @>{\partial}>> K_0(C_\pi^*\widetilde M) 
     \xrightarrow{\iso} K_0(C^*_{red}\pi)\\
@VV{tr}V  @VV{tr}V   @VV{tr}V\\
   K^{lf}_0(\mathcal{M}) @>{\iso}>{PD}> K_1(D_{\Gamma}^*\widetilde{\mathcal{M}}/C_{\Gamma}^*\widetilde{\mathcal{M}})
   @>{\partial}>> K_0(C_{\Gamma}^*\widetilde {\mathcal{M}})\\
 @VV{[\mathcal{L}] \cap -}V  && @VV{\psi}V\\
    K^{lf}_0 ({\mathcal{M}},
C^*_{red} \Gamma
) @>>> KX_0(\mathcal{M},
C^*_{red} \Gamma
) @>>{\mu_\infty}>
    K_0(C^*(\mathcal{M},
C^*_{red} \Gamma
))\\
@VV{\phi_*}V @VV{\phi_*}V @VV{\phi_*}V\\
K^{lf}_0(\mathcal{B}^n,
C^*_{red} \Gamma
) @>{cX}>> KX_0(\mathcal{B}^n,
C^*_{red} \Gamma
) @>{\iso}>{\mu_\infty}>
 K_0(C^*(\mathcal{B}^n,
C^*_{red} \Gamma
)) \ .
  \end{CD}
\end{equation*}
The horizontal arrows denoted PD are Paschke duality isomorphism, compare
e.g.~\cite[Section 3]{Roean} and~\cite{HR}. The vertical maps $tr$ are
transfer maps, on the level of $C^*$ and $D^*$ they are simply given by
diagonal embedding. The map $\partial$ is a boundary map in a long exact $K$-theory
sequence, the compositions of $tr$ and $\partial$ are Baum--Connes assembly maps
$\mu$ (compare~\cite{Roean} again). Finally, $cX$ is the coarsening map from
locally finite to coarse K-homology.

In order to detect the nonvanishing of $\alpha(M)$ claimed by Theorem~\ref{t:main}, 
we chase  the $K$-theoretic fundamental class $[M]_K \in K_0(M)$ through this diagram.  
Because the coarse assembly map $\mu_{\infty}$ is an isomorphism for the space $\mathcal{B}^n$ 
by Proposition~\ref{p:coarseass}, we only need to show that the image of $[M]_K$ under the map 
\[
     \omega = cX\circ \phi_*\circ ([\mathcal{L}]\cap -)\circ tr\colon  K_0( M ) \longrightarrow KX_0(\mathcal{B}^n ;
C^*_{red} \Gamma 
) 
\]
is non-zero. This will ultimately follow from a form of Atiyah's  
$L^2$-index theorem. 

For a metric space $X$ and a $C^*$-algebra $A$, we define the $K$-homology with 
compact supports and coefficients in $A$ as 
\[
   KKR(X;A) := \lim_{Q \subset Z} KK(C_0(Q) ; A) \, , 
\]
where the limit goes over the set of compact subsets of $X$ ordered by inclusion. 
Note in particular that $KKR(X;\C) = K_0(X)$, if $X$ is homotopy 
equivalent to a $CW$-complex. The canonical inclusions 
$\C \to C_0(Q)$ induce an augmentation map 
\[
    \epsilon \colon  KKR(X;A) \to KK(\C ; A) \, . 
\]
 
Now let $E \to S^n$ be a finite dimensional unitary bundle so that (with $k := \dim E$)
the virtual bundle $E - \underline{\C}^k$ 
represents a generator of  $K^0(S^n,S)$ compatible with the orientation used for 
defining the $K$-theoretic  degrees $z_i$ at the beginning of this section.
(Recall that $n$ is even by assumption.) 
We define a map 
\[
   KK(C_0(\mathcal{B}^n) ; 
C^*_{red} \Gamma
 ) = \prod_i KK(C_0(\underline{S}^n(i)) ; 
C^*_{red} \Gamma 
) \lra \prod_i K_0( 
C^*_{red} \Gamma) \lra \prod_i \R
\]
where each component of the last map is the composition
\[
    KK(C_0( \underline{S}^n ) ; C^*_{red} H ) 
     \stackrel{- \cap [E - \underline{\C}^k]}{\lra}  
     KKR (\underline{S}^n  ; 
C^*_{red} \Gamma 
) \stackrel{\epsilon}{\lra} 
    KK(\C ; 
C^*_{red} \Gamma  
) \xrightarrow{\tau}\R \, . 
\]
Here, the first map is given by the cap product with the $K$-cohomology class with
compact support represented by the virtual bundle $[E -\underline{\C}^k]$. The last map is 
induced by the canonical trace $\tau\colon C^*_{red}\Gamma\to\C$.

\begin{lem} 
The composition 
 \[
   K_0^{lf}(\mathcal{B}^n ; C^*_{red} \Gamma  ) \lra \prod_i K_0(C^*_{red} \Gamma) 
   \lra \prod_i \R \lra (\prod_i \R) / (\bigoplus_i \R)
\] 
factors through the canonical map 
\[
    K_0^{lf}(\mathcal{B}^n ; C^*_{red} \Gamma  ) \lra KX_0 ( \mathcal{B}^n ; C^*_{red} \Gamma  ) \, . 
\]         
\end{lem} 
This is an immediate consequence of the calculation appearing in the proof of Proposition~\ref{p:coarseass}.

The following proposition concludes the proof that $\alpha(M) \neq 0$. 
\begin{prop}\label{p:last}
The composition 
\[
    K_0( M ) \stackrel{\omega}{\lra} 
    KX_0 (\mathcal{B}^n ; C^*_{red} \Gamma  ) \lra (\prod_i \R) / (\bigoplus_i \R)
\]
sends the $K$-theoretic fundamental class of $M$ to the element represented by the sequence
$(z_1, z_2, \ldots )$. 
\end{prop}

Recall that $z_i \in \Z$ is  defined as the $K$-theoretic degree of the map 
\[
   f_{1/i}  \colon M_{1/i} \lra S^n 
\]
and non-zero by assumption. Hence, the sequence given 
in Proposition~\ref{p:last} represents a non-zero class in the quotient $(\prod \R) / (\bigoplus \R)$ 
and the proof of Theorem~\ref{t:main} is complete modulo the proof of Proposition~\ref{p:last}.   

To prove this Proposition, we  first need to recall a form of Atiyah's 
$L^2$-index theorem. Let $\Gamma$ be an arbitrary (countable) discrete
group. We consider the composition
\[
   K_*(B \Gamma) \stackrel{\mu}{\lra} K_*(C^*_{red} 
   \Gamma) \stackrel{\tau}{\lra} \R   \, , 
\]
where the first map is the Baum--Connes assembly map and the second map 
is induced by the canonical trace $\tau \colon   C^*_{red} \Gamma \to \C$. 
Now the homological form of the $L^2$-index theorem reads as follows. The
elegant proof in~\cite{CM} applies without change. 

\begin{prop}\label{p:atiyah} 
The composition $\tau_* \circ \mu$ is equal to the 
map 
\[
    K_*(B\Gamma) \longrightarrow K_*(pt.)\iso\Z \hookrightarrow \R,
\]
induced by mapping $B\Gamma$ to a point (recall that $K_*(B\Gamma)$ denotes
$K$-homology with compact support). Of course, given a (not necessarily
compact) manifold $X$ and a map $c\colon X\to B\Gamma$, $\tau_*\circ\mu\circ
c_*\colon K_*(X)\to K_*(\C)$ is then equal to the map induced by the
projection $X\to pt.$
\end{prop}

\begin{proof}[Proof of Proposition~\ref{p:last}]
It is enough to show that under the composition 
\[
      K_0 ( M ) \lra KK(C_0 (\mathcal{B}^n)  ; C^*_{red} \Gamma  ) 
     \lra \prod_i K_0( C^*_{red} \Gamma )  \lra \prod_i \R
\]
the fundamental class is sent to the sequence $(z_1, z_2, z_3, \ldots)$. This will 
follow from the fact that the composition 
\begin{eqnarray*}
    K_0(M) & \lra & KK ( C_0 (M_{1/i}) ; C^*_{red} \Gamma)  \\
          & \stackrel{(f_{1/i} )_*}{\lra} & KK ( C_0 (\underline{S}^n) ; C^*_{red} \Gamma) \\    
          & \stackrel{- \cap [E - \underline{\C}^k]}{\lra} & KKR( \underline{S}^n  ; C^*_{red} \Gamma) \\
          & \stackrel{\epsilon}{\lra} & KK(\C ; C^*_{red} \Gamma) \stackrel{\tau}{\longrightarrow}  \R  \\
\end{eqnarray*}
sends the $K$-theoretic fundamental class to $z_i$ (the first map is a 
composition of the transfer map and the slant product with the Mishchenko 
bundle). 

First we observe that  the preceding composition is equal to the composition 
\begin{eqnarray*}
   K_0(M) & \stackrel{tr}{\longrightarrow} & K_0^{lf}(M_{1/i}) = KK(C_0(M_{1/i}) ; \C) \\ 
          & \stackrel{- \cap f_{1/i}^*([E - \underline{\C}^k]) }{\longrightarrow} & K_0(M_{1/i} ) \\
          & \stackrel{- \cap [L]}{ \longrightarrow } & KKR(M_{1/i}; C^*_{red} \Gamma) \\
          & \stackrel{\epsilon}{\longrightarrow} & KK(\C ; C^*_{red} \Gamma)
          \stackrel{\tau}{\longrightarrow} \R \, .
\end{eqnarray*} 

Now the composition 
\[
  K_0(M_{1/i} ) \stackrel{- \cap [L]}{\lra} KKR(M_{1/i} ; C^*_{red} \Gamma) \stackrel{\epsilon}{\lra} KKR(\C ; C^*_{red} \Gamma) 
  = K_0(C^*_{red} \Gamma_i) 
\]
is (by one possible definition of $\mu$) equal to the composition 
\[
  K_0(M_{1/i}) \stackrel{c_*}{\lra} K_0(B\Gamma) \stackrel{\mu}{\lra} K_0(C^*_{red} \Gamma) \, . 
\]
Here $c$ is the composition
\[
    c \colon M_i \lra B\Gamma_i\to B\Gamma 
\]
where the first map classifies the universal cover of $M_i$ and the second is
induced from the canonical inclusion.

Therefore, using Proposition~\ref{p:atiyah}, we need 
to show that the composition 
\[
\begin{CD}
  K_0(M) @>{tr}>> K_0^{lf} (M_{1/i}) @>{- \cap
    f_{1/i}^*([E - \underline{\C}^k]) }>> K_0(M_{1/i} ) @>{\epsilon}>> K_0(pt. ) =
  \Z\\
  && @VV{(f_{1/i})_*}V @VV{(f_{1/i})_*}V @VV{=}V\\
  && K_0(S^n,S) @>{_\cap [E-\underline{\C}^k]}>> K_0(S^n) @>{\epsilon}>> K_0(pt.)
  \end{CD}
\]
sends the $K$-theoretic fundamental class of $M$ to $z_i$. But this assertion is 
immediate. 
\end{proof}


\bigskip

\begin{thebibliography}{AAAA}

\bibitem{MFA}
M.~F.~Atiyah, {\it Elliptic operators, discrete groups and von Neumann algebras}, Ast\'erisque {\bf 32-3} (1976), 43--72.

\bibitem{BCH} 
P.~Baum, A.~Connes, N.~Higson, {\it Classifying space for proper $G$-actions and K-theory 
of group $C^*$-algebras}, in {\sl Proceedings of a special session on $C^*$-algebras}, Contemp. Math. 
{\bf 167} (1994), 241--291. 

\bibitem{Bl} 
B.~Blackadar, {\sl K-theory for operator algebras}, Second Edition, Cambridge University Press 1998. 

\bibitem{BW}
J.~Block, S.~Weinberger, {\it Aperiodic tilings, positive scalar curvature and amenability of spaces}, 
J.~Amer.~Math.~Soc.~{\bf 5} (1992), 907--918. 

\bibitem{CM} 
I.~Chatterji, G.~Mislin, {\it Atiyah's $L^2$-index theorem}, Enseign.~Math.~{\bf 49} (2003), 85-93.  

\bibitem{D}
A.~N.~Dranishnikov, {\it On hypersphericity of manifolds of finite asymptotic dimension}, Trans.~Amer.~Math. 
Soc.~{\bf 355} (2002), 155--167.

\bibitem{dwyer}
W.~Dwyer, T.~Schick, S.~Stolz, {\it Remarks on a conjecture of Gromov and Lawson}, in 
{\sl High-dimensional manifold topology}, 159--176, World Sci.~Publishing, River Edge, NJ, 2003.

\bibitem{Gong}
G.~Gong, G.~Yu, {\it Volume growth and positive scalar curvature}, Geom. funct. anal. {\bf 10} (2000), 821--828.

\bibitem{BC}
M.~Gromov, {\it Volume and bounded cohomology}, 
Publ.~Math.~I.H.E.S.~{\bf 56} (1982), 5--99.

\bibitem{Large}
M.~Gromov, {\it Large Riemannian manifolds}, 
Springer LNM~{\bf 1201} (1985), 108--122.

\bibitem{AI}
M.~Gromov, {\it Asymptotic invariants of infinite groups}, 
in {\sl Geometric Group Theory}, vol.~2, London Math.~Soc.~Lecture Note Series {\bf 182}, 
Cambridge Univ.~Press 1993.

\bibitem{Macro}
M.~Gromov, {\it Positive curvature, macroscopic dimension, spectral gaps, and higher signatures}, 
in {\sl Functional Analysis on the Eve of the 21st Century}, vol.~II, Prog.~in Math.~{\bf 132}, 
Birkh{\"a}user 1996.

\bibitem{GL} 
M.~Gromov, H.~B.~Lawson, {\it Spin and scalar curvature in the presence of a fundamental group I}, 
Ann.~Math.~{\bf 111} (1980), 209--230.

\bibitem{GL2} 
M.~Gromov, H.~B.~Lawson, {\it Positive scalar curvature and the Dirac operator on complete
Riemannian manifolds}, Publ.~Math.~I.H.E.S.~{\bf 58} (1983), 83--196.

\bibitem{HS}
B.~Hanke, T.~Schick, {\it Enlargeability and index theory}, J.~Differential Geom.~{\bf 74} (2006), 293--320.

\bibitem{HS2} 
B.~Hanke, T.~Schick, {\it Enlargeability and index theory: infinite covers}, $K$-Theory {\bf 38} (2007), 23--33. 

\bibitem{HPR}
N.~Higson, E.~Pedersen, J.~Roe, {\it $C^{*}$-algebras and controlled topology},
$K$-Theory {\bf 11} (1997), {209--239}.

\bibitem{HR}
N.~Higson, J.~Roe, {\sl Analytic $K$-Homology}, Oxford University Press 2000.

\bibitem{KKS} 
D.~Kahn, J.~Kaminker, C.~Schochet, {\it Generalized homology theories 
on compact metric spaces}, Mich.~Math.~J.~{\bf 24} (1977), 203--224.

\bibitem{KS} 
J.~Kaminker, C.~Schochet, {\it $K$-theory and Steenrod homology: Applications 
to the Brown--Douglas--Fillmore theory of operator algebras}, Trans.~Amer.~Math.~Soc.~{\bf 227} 
(1977), 63--107.

\bibitem{Milnor} 
J.~Milnor, {\it On the Steenrod homology theory} (first distributed 1961), 
in: S.~Ferry, A.~Ranicki and J.~Rosenberg (eds.), {\sl Novikov conjectures, index theorems 
and rigidity}, Volume 1, Oberwolfach 1993, LMS Lecture Notes {\bf 226}, 79--96.

\bibitem{R} 
J.~Roe, {\sl Index Theory, Coarse Geometry, and Topology
of Manifolds}, CBMS Regional Conference Series in Mathematics {\bf 90}, 
Amer.~Math.~Soc.~1996. 

\bibitem{Roean}
J.~Roe, {\it Comparing analytic assembly maps},
Quart.~J.~Math.~{\bf 53} (2002), 241--248.

\bibitem{Roe} 
J.~Roe, {\sl Lectures on Coarse Geometry}, University Lecture Series  {\bf 31}, 
Amer.~Math.~Soc.~2003. 

\bibitem{Schick} 
T.~Schick, {\it A counterexample to the (unstable) Gromov--Lawson--Rosenberg conjecture}, 
Topology {\bf 37} (1998), 1165--1168.

\bibitem{Schu} 
R.~Schultz (ed.), {\sl Group Actions on Manifolds},
(Boulder, Colorado, 1983), Contemp. Math. {\bf 36}, Amer.~Math.~Soc.

\bibitem{Ststable} 
S.~Stolz, {\it Manifolds of positive scalar curvature}, in: T.~Farrell et 
al.~(eds.), {\sl Topology of high dimensional manifolds} (Trieste 2001), ICTP Lecture 
Notes, vol. 9, 661--709. 

\bibitem{WO} 
N.~Wegge-Olsen, {\sl $K$-Theory and $C^*$-algebras}, Oxford University Press,

\bibitem{Yu} 
 G.~Yu, {\it The coarse Baum--Connes conjecture for spaces which admit 
 a uniform embedding into Hilbert space}, Invent.~Math.~{\bf 139} (2000), 201--240.

\end{thebibliography}
\end{document}